\documentclass{amsart}

\newtheorem{theorem}{Theorem}[section]
\newtheorem{lemma}[theorem]{Lemma}

\theoremstyle{definition}
\newtheorem{definition}[theorem]{Definition}
\newtheorem{example}[theorem]{Example}

\usepackage{enumerate}
\theoremstyle{remark}

\numberwithin{equation}{section}

\begin{document}

\title[SGL submanifolds in an indefinite Sasakian statistical manifold]{The geometric characteristics of SGL submanifolds in an indefinite Sasakian statistical manifold equipped with a quarter symmetric metric connection} 

\author{Vandana Gupta}
\address{Department of Mathematics, Patel Memorial National College, Rajpura,  Patiala, India}

\email{vgupta.87@rediffmail.com}

\author{Shagun}
\address{Department of Mathematics, Punjabi University, Patiala, India}
%    Information for third author
\author{Jasleen Kaur }
\address{Department of Mathematics, Punjabi University, Patiala, India}
\email{jasleen\_math@pbi.ac.in }

\begin{abstract}

This research paper examines the geometric structure of screen generic lightlike (SGL) submanifolds in an indefinite Sasakian statistical manifold equipped with a quarter-symmetric (QS) metric connection. The study focuses on analyzing the integrability conditions and the parallelism properties of various distributions associated with these submanifolds. It explores the characteristics of totally geodesic foliations  and mixed geodesic submanifolds, providing significant insights into their geometric behavior. In addition to the theoretical development, the paper also presents an illustrative example of a contact SGL submanifold within an indefinite Sasakian statistical manifold.
\end{abstract}

\maketitle

\section{\textbf{INTRODUCTION}}

%Lightlike theory of a semi-Riemannian manifold is an interesting branch of study in the field of geometry. The concept of contact $CR$ and $SCR$-lightlike submanifolds in the context of indefinite Sasakian manifolds were introduced by Duggal and Bejancu \cite{duggal1996lightlike}. These submanifolds were further analyzed by researchers for various indefinite contact metric manifolds \cite{sasakianduggal}, \cite{A} and so on. Additionally, \cite{duggalgeneric}, \cite{B} introduced the notion of generic lightlike submanifold, which is an extension of geometry of the half lightlike submanifold of codimension 2, for an indefinite Sasakian  manifold. In this regard, to further generalise this concept, \cite{D} initiated the idea of screen generic lightlike submanifolds (Through out the paper, we write ``screen generic lightlike submanifolds" as ``SGL submanifolds") which encompasses SCR lightlike submanifolds and generic lightlike submanifolds within the framework of an indefinite Sasakian manifold and worked upon the severval properties related to their structure. 
 Duggal and Bejancu \cite{duggal1996lightlike} introduced contact 
$CR$ and $SCR$ lightlike submanifolds in the setting of indefinite Sasakian manifolds, which laid the groundwork for further studies. Subsequent research extended the analysis of these submanifolds within various types of indefinite contact metric manifolds \cite{sasakianduggal}, \cite{A}. The notion of generic lightlike submanifolds, which generalizes the geometry of half-lightlike submanifolds of codimension 2, was introduced in \cite{duggalgeneric}, \cite{B} for indefinite Sasakian manifolds. To further extend this concept, the idea of screen generic lightlike submanifolds (henceforth referred to as “SGL submanifolds” throughout this paper) was introduced in \cite{D}. This novel class of submanifolds serves as a broader framework that includes both SCR-lightlike submanifolds and generic lightlike submanifolds within indefinite Sasakian geometry. \cite{D} focused on various structural properties of SGL submanifolds, providing deeper insights into their geometric characteristics.

 Analyzing geometric structures within sets of certain probability distributions results in the development of statistical manifolds, whose study was pioneered by \cite{E} and followed by significant contributions from many researchers \cite{F}, \cite{G}, \cite{H} et al. Thereafter, the lightlike theory of statistical manifolds has been explored by \cite{K}, \cite{bahadir2019geometry} and many others. Further, several observations have been presented related to the lightlike theory of various contact metric statistical manifolds by amalgamating statistical structure with an indefinite contact metric structure in \cite{I}, \cite{L} et.al.\\
 
Golab \cite{golab1975} introduced a quarter-symmetric connection in a differentiable manifold. Hereafter, quarter-symmetric metric connection will be abbreviated as "QS" metric connection. 
 ``A linear connection $\tilde{D}$ on a semi-Riemannian manifold $(\tilde{ N}, \tilde{\rho})$ is said to be a quarter symmetric connection if its torsion tensor $\tilde{T}$ of the connection $\tilde{D}$ satisfies 
\[
\tilde T(X,Y) = \tilde{D}_XY -\tilde{D}_YX -[X,Y]
\]
\begin{equation}\label{eq1.1}
	\tilde T(X,Y) = \eta(Y)\phi(X) - \eta(X)\phi(Y),
\end{equation}
where $\phi$ is a (1,1)-tensor field and $\eta$ is a 1-form defined by $\eta(X) =\tilde \rho(X,\nu) $. If the linear connection $\tilde D$  satisfies $\tilde D \rho = 0$, it is called a QS metric connection." Subsequently, extensive research has been conducted on lightlike submanifolds within various semi-Riemannian manifolds, including indefinite trans-Sasakian manifolds and indefinite Kähler manifolds, in the context of a QS metric connection(refer to\cite{M} and \cite{N}).

Building on this theoretical foundation, the present study aims to develop the framework of SGL submanifolds within an indefinite Sasakian statistical manifold endowed with a QS metric connection. The paper examines the conditions of parallelism  and integrability  of different distributions. Furthermore, significant results concerning totally geodesic foliations and mixed geodesic submanifolds are established.

\vspace{.2in}
\section{Preliminaries}
The study of lightlike submanifolds in semi-Riemannian geometry necessitates a thorough understanding of several fundamental concepts that define their geometric characteristics. This section explores some such essential ideas.
\begin{definition}\label{def21}
`` A pair $(\bar{\nabla}, \tilde{\rho})$ is called a \textbf{statistical structure} on a semi-Riemannian manifold $\tilde{N}$ such that for all $X, Y, Z \in \Gamma(T\tilde{N})$
 \begin{enumerate}
 	\item $\bar{\nabla}_{X}Y-\bar{\nabla}_{Y}X = [X,Y]$;
 	\item $(\bar{\nabla}_{X} \tilde{\rho})(Y, Z) = (\bar{\nabla}_{Y}\tilde{\rho})(X, Z)$ hold.
 \end{enumerate}   
 Then $(\tilde{N}, \tilde{\rho}, \bar{\nabla})$ is said to be an \textbf{indefinite statistical manifold}. Moreover, there exists $\bar{\nabla}^{*}$ which is a dual connection of $\bar{\nabla}$ with respect to $\tilde{\rho}$, satisfying
 \begin{equation}\label{eq2.1}
 X\tilde{\rho}(Y, Z) = \tilde{\rho}(\bar{\nabla}_{X}Y, Z) + \tilde{\rho}(Y, \bar{\nabla}^{*}_{X}Z).
 \end{equation}
 
 Also $(\bar{\nabla}^*)^* = \bar{\nabla}$. If $(\tilde{N}, \tilde{\rho}, \bar{\nabla})$ is an indefinite statistical manifold, then   $(\tilde{N}, \tilde{\rho}, \bar{\nabla}^{*})$ is also a statistical manifold. Hence, the indefinite statistical manifold is denoted by $(\tilde{N}, \tilde{\rho}, \bar{\nabla}, \bar{\nabla}^{*})$."\\
 \end{definition}
 
  Consider ($ \tilde{N},\tilde \rho$) as an $(m+n)$-dimensional semi-Riemannian manifold with semi-Riemannian metric $\tilde{\rho}$ and of constant index $q$ such that $m,n\geq 1$, $1\leq q\leq m+n-1$.\\
`` Let $(N,\rho)$ be a $m$-dimensional lightlike submanifold of $\tilde{N}$. In this case, there exists a smooth distribution $Rad(TN)$ on $N$ of rank $r>0$, known as Radical distribution on $N$ such that $Rad (TN_p) = TN_p \cap TN_p^{\perp}, \forall ~p \in N$ where $TN_p$ and  $TN_p^{\perp}$ are degenerate orthogonal spaces but not complementary. Then $N$ is called an $r$-lightlike submanifold of $\tilde{N}$.  Now, consider $S(TN)$, known as screen distribution, as  a complementary distribution of radical distribution in  $TN$  i.e.,
$ TN = Rad (TN)  \perp S(TN)$ and  $S(TN^{\perp})$, called screen transversal vector bundle, as a complementary vector subbundle to $Rad(TN)$ in $TN^{\perp}$ i.e.,
 $ TN^{\perp} = Rad(TN)  \perp S(TN^{\perp})$. As $S(TN)$ is non degenerate vector subbundle of $T\tilde{N}{\mid}_N$, we have
$ T\tilde{N}{\mid}_N = S(TN) \perp S(TN)^{\perp}$
 where $S(TN)^{\perp}$  is the complementary orthogonal vector subbundle of $S(TN)$ in $T\tilde{N}{\mid}_N$. Let $tr(TN)$ and $ltr(TN)$ be complementary vector bundles to $TN$ in 
 $T\tilde{N}{\mid}_N$ and to $Rad(TN)$ in $S(TN^{\perp})^{\perp}$. Then we have $
 tr(TN) = ltr(TN) \perp S(TN^{\perp})$, $ T\tilde{N}{\mid}_N = TN \oplus tr(TN)= (Rad(TN) \oplus ltr(TN)) \perp S(TN) \perp S(TN^{\perp})$.
 
 \begin{theorem}\cite{duggal1996lightlike}
 Let $(N,\rho,S(TN), S(TN^{\perp}))$ be an $r$- lightlike submanifold of a semi-Riemannian manifold $(\tilde{N},\tilde{\rho})$.Then there exists a complementary vector bundle $ltr(TN)$ called a lightlike transversal bundle of $Rad(TN)$ in $S(TN^{\perp})^{\perp}$ and basis of $\Gamma(ltr(TN){\mid}_U)$ consisting of smooth sections $\{N^{\prime}_1,\cdots,N^{\prime}_r\}$ $S(TM^{\perp})^{\perp}{\mid}_U$ such that 
 \[
 \bar{g}(N^{\prime}_i,\xi_j) = \delta_{ij} , \quad \bar{g}(N^{\prime}_i,N^{\prime}_j) = 0, \quad i,j=0,1,\cdots , r
 \]
 where  $\{{\xi_1, \cdots , \xi_r}\}$ is a lightlike basis of $\Gamma(RadTM){\mid}_U$."
 \end{theorem}

 Let $(N,\rho)$ be a lightlike submanifold of  $(\tilde{N},\tilde{ \rho},\bar{\nabla},\bar{\nabla^*})$. Then the G-W formulae are as below:
  \begin{equation}\label{eq2.2}
  \bar\nabla_X Y = \nabla_X Y +h^l(X,Y) + h^s(X,Y) , \quad \bar\nabla^*_X Y = \nabla^*_X Y +h^{*l}(X,Y) +h^{*s}(X,Y),   
  \end{equation}
  \begin{equation}\label{eq2.3}
  \bar\nabla_X V = -A_V X + D_X^{l}V +D_X^{s}V , \quad \bar\nabla_X^* V = -A_V^* X + D_X^{*l}V +D_X^{*s}V ,
  \end{equation}
  \begin{equation}\label{eq2.4}
 	\bar\nabla_X N^{\prime} = -A_{N^{\prime}} X + \nabla^l_X N^{\prime} +D^s(X,N^{\prime}), \quad \bar\nabla_X^* N^{\prime} = -A^*_{N^{\prime}} X + \nabla^{*l}_X N^{\prime} +D^{*s}(X,N^{\prime}),
  \end{equation}
  \begin{equation}\label{eq2.5}
  	\bar\nabla_X W =-A_W X +\nabla^s_X W +D^l(X,W),\quad \bar\nabla_X^*W =-A_W^*X + \nabla^{*s}_XW +D^{*l}(X,W).
  \end{equation}
  for any  $X, Y \in \Gamma(TN)$, $V \in \Gamma(tr(TN))$, $N^{\prime} \in \Gamma(ltr(TN))$ and $W\in \Gamma(S(TN^{\perp}))$.\\
  
The notion of indefinite statistical manifold and (\ref{eq2.2}), (\ref{eq2.3}), (\ref{eq2.4}), (\ref{eq2.5}) lead to
  \begin{equation}\label{eq2.6}
 \tilde{\rho}(h^s(X,Y),W) + \tilde{ \rho}(Y, D^{*l}(X,W)) = \tilde{ \rho}(Y,A^*_WX),
 \end{equation}
 \[
 \tilde{\rho}(h^l(X,Y),\xi) + \tilde{\rho}(Y,\nabla^*_X\xi) + \tilde{ \rho}(Y,h^{*l}(X,\xi))=0,
 \]
 Within statistical manifold theory, non-degenerate submanifolds are known to preserve the statistical structure. Nevertheless, this assertion breaks down for lightlike submanifolds, as the definition of a statistical manifold and equation (\ref{eq2.2}) implies
  \[
  (\nabla_X\rho)(Y,Z) - (\nabla_Y\rho)(X,Z) = \tilde{\rho}(Y,h^l(X,Z)) - \tilde{\rho}(X,h^l(Y,Z)),
  \]
  and   
  \[
  X\rho(Y,Z)- \rho(\nabla_XY,Z) - \rho(Y, \nabla^*_XZ) = \tilde{ \rho}(h^l(X,Y),Z) + \tilde{\rho}(Y,h^{*l}(X,Z)).
  \]
  Considering the projection morphism $P$ of  $TN$ to $S(TN)$, we have 
  \begin{equation}\label{eq2.7}
  \nabla_X PY = \nabla_X^{\prime} PY + h^{\prime}(X,PY), \quad \nabla_X^* PY = \nabla_X^{*\prime} PY + h^{*\prime}(X,PY),
  \end{equation}
  \begin{equation}\label{eq2.8}
  \nabla_X \xi = - A^{\prime}_{\xi} X + \nabla_X^{\prime t} \xi, \quad \nabla_X^* \xi = - A^{*\prime}_{\xi} X + {\nabla}_X^{*\prime t} \xi,
  \end{equation}
  for any $X, Y \in \Gamma(TN)$,$\;$ $\xi \in \Gamma(Rad(TN))$. \\
  Using (\ref{eq2.2}),(\ref{eq2.3}),(\ref{eq2.6}) and (\ref{eq2.8}), we obtain
  \begin{equation}\label{eq2.9}
  \tilde \rho(h^l(X,PY),\xi) = \rho(A^{*\prime}_\xi X,PY), \quad \tilde \rho(h^{*l}(X,PY),\xi) = \rho(A^{\prime}_\xi X,PY),
  \end{equation}
  \begin{equation}\label{eq2.10}
  \tilde \rho(h^{\prime}(X,PY),N^{\prime}) = \rho(A^*_{N^{\prime}}X,PY), \quad \tilde \rho(h^{*\prime}(X,PY),N^{\prime}) = \rho(A_{N^{\prime}}X,PY),
  \end{equation}
  for any $X, Y \in \Gamma(TN)$, $\xi \in \Gamma(Rad(TN))$ and $N^{\prime} \in \Gamma(ltr(TN))$. As $h^l$ and $h^{*l}$ are symmetric, so from (\ref{eq2.9}), we obtain
  \[
 \rho(A^{\prime}_\xi PX,PY) =\rho( PX,A^{\prime}_\xi PY), \quad \rho(A^{*\prime}_\xi PX,PY) = \rho( PX,A^{*\prime}_\xi PY).
  \]

 `` Let $\bar{\nabla}^{\circ}$ be the Levi-Civita connection w.r.t $\tilde{ \rho}$. Then, we have $\bar{\nabla}^{\circ} = \frac{1}{2}(\bar{\nabla}+ \bar{\nabla}^*).$ \\
   For a statistical manifold $(\tilde{N},\tilde{\rho},\bar{\nabla},\bar{\nabla}^*)$, the difference $(1,2)$ tensor $K$ of a torsion free affine connection $\bar{\nabla}$ and Levi-Civita connection $\bar{\nabla}^{\circ}$ is defined as
   \begin{equation}\label{eq2.11}
   K(X,Y) = K_XY= \bar{\nabla}_X Y - \bar{\nabla}^{\circ}_X Y,
   \end{equation}
   Since $\bar{\nabla}$ and $\bar{\nabla}^{\circ}$ are torsion free, we have
   \begin{equation}\label{eq2.12}
   K(X,Y) = K(Y,X),\;\;\;\;\; \tilde \rho(K_XY,Z) = \tilde \rho(Y,K_XZ),
   \end{equation}
   for any $X, Y, Z \in \Gamma(TN)$."
    
     \begin{definition}\cite{sasakianduggal} ``An odd-dimensional semi-Riemannian manifold $(\tilde{N}, \tilde{\rho})$ is called contact metric manifold if there are a $(1,1)$ tensor field $\phi$, a vector field $\nu$ called characteristic vector field, and a 1-form $\eta$ such that
         	\begin{equation}\label{eq2.13}
         		\tilde{\rho}(\phi X, \phi Y) = \tilde{\rho}(X, Y) - \eta(X)\eta(Y), \qquad \tilde{\rho}(\nu, \nu) = 1,   
         	\end{equation}
         	\begin{equation}\label{eq2.14}
         		\phi^{2}(X) = - X + \eta(X)\nu, \quad \tilde{\rho}(X, \nu) = \eta(X),\quad \tilde{\rho}(\phi X, Y)  + \tilde{\rho}(X, \phi Y) = 0,
         	\end{equation}
         	\begin{equation}\label{eqeq2.15}
         		d\eta(X,Y) = \tilde{\rho}(X,\phi Y), \quad \forall \;X,Y \in \Gamma(T\bar{M}) .
         	\end{equation}
         	It follows that $\phi\nu = 0,\; \eta o\phi = 0,\; \eta(\nu) = 1$. Then $(\phi, \nu, \eta, \tilde{\rho})$ is called contact metric structure of $\tilde{N}$.\\
         	Also, $\tilde{N}$ has a normal contact structure if $N_{\phi} + d\eta \otimes \nu = 0$, where $N_{\phi}$ is the Nijenhuis tensor field."
         \end{definition}
         
         \noindent ``A normal contact metric manifold $\tilde{N}$ is called an indefinite Sasakian manifold if
         \begin{equation}\label{eq2.16}
         	{\bar{\nabla}^{\circ}}_{X}\nu = - \phi X, 
         \end{equation}
         \begin{equation}\label{eq2.17}
         	({\bar{\nabla}^{\circ}}_{X}\phi)Y = \tilde{\rho}(X,Y)\nu - \eta(Y)X.
         \end{equation}
         holds for any $X, Y \in \Gamma(T\tilde{N})$, where ${\bar{\nabla}^{\circ}}$ is Levi-Civita Connection. " 
\begin{definition}\cite{furuhata2017}\label{def2.4}
``	A quadruplet $(\bar{\nabla} = {\bar{\nabla}^{\circ}} + K, \tilde{\rho}, \phi, \nu)$ is called  Sasakian statistical structure on $\tilde{N}$ if\\ 
	$(i)\;(\tilde{\rho},\phi,\nu)$ is  Sasakian structure on $\tilde{N}$\\
	$(ii)\;(\bar{\nabla},\tilde{\rho})$ is a statistical structure on $\tilde{N}$\\
	and the condition
	\begin{equation}\label{eq2.18}
	K(X, \phi Y) + \phi K(X, Y) =0
	\end{equation}
	holds for any $X, Y \in \Gamma(T\tilde{N})$.\\
	Then $(\tilde{N}, \bar{\nabla}, \tilde{\rho},\phi, \nu)$ is called Sasakian statistical manifold. If $(\bar{N}, \bar{\nabla}, \tilde{\rho}, \phi, \nu)$ is  Sasakian statistical manifold, then so is $(\tilde{N}, \bar{\nabla}^{*}, \tilde{\rho}, \phi, \nu)$."
\end{definition}
\begin{theorem}\cite{furuhata2017}
``	Let $(\tilde{N}, \bar{\nabla}, \tilde{\rho})$ be an indefinite statistical manifold and $(\tilde{\rho}, \phi, \nu)$ an almost contact metric structure on $\bar{N}$. Then $(\bar{\nabla}, \tilde{\rho}, \phi, \nu)$ is an indefinite  Sasakian statistical structure if and only if the following conditions hold:
	
	\begin{equation}\label{eq2.19}
	\bar{\nabla}_{X}\phi Y - \phi\bar{\nabla}^{*}_{X}Y =   \tilde{\rho}(X,Y)\nu - \eta(Y)X,
	\end{equation}
	\begin{equation}\label{eq2.20}
	\bar{\nabla}_{X}\nu = -\phi X + \tilde \rho(\bar{\nabla}_{X}\nu, \nu)\nu,
	\end{equation}	
	for the vector fields $X, Y$ on $\tilde{N}$."	
\end{theorem}
     
\section{Contact SGL Submanifolds}
Dogan et al.\cite{dogan2019} introduced a new category of submanifolds known as SGL submanifolds, which encompasses invariant lightlike, screen real lightlike, and generic lightlike submanifolds. In this perspective, the notion of contact SGL submanifolds has been formulated for an indefinite Sasakian statistical manifold equipped with a QS-metric connection.
\begin{definition}\label{def31}
``A real lightlike submanifold $N$ of an indefinite  Sasakian statistical manifold $\tilde{N}$  is a SGL submanifold if 
      	 \begin{enumerate}
\item   $Rad(TN)$ is invariant respect to ${\phi}$, that is,
\[
\phi(Rad(TN)) = Rad(TN).
\]   
\item There exists a subbundle $E_{\circ}$ of $S(TN)$ such that
\[
E_{\circ} = \phi(S(TN)) \cap S(TN)
\]
where $E_{\circ}$ is a non-degenerate distribution on $N$.	 
      	 \end{enumerate}
From definition of a SGL submanifold, there exists a complementary non-degenerate distribution $E^{\prime}$ to  $E_{\circ}$ in $S(TN)$ such that
\[
S(TN) = E_{\circ} \oplus E^{\prime}\perp \nu
\] 
where
\[
\phi(E^{\prime}) \not\subseteq S(TN) \quad and \quad \phi(E^{\prime}) \not\subseteq S(TN^{\perp})."
\]    	 
\end{definition}
Let $P_{\circ}$, $P_1$ and $Q$ be the projection morphisms on $E_{\circ}$, $Rad(TN)$ and $E^{\prime}$, respectively.\\

Then, for $X \in \Gamma(TN)$,
\begin{align}
X &= P_{\circ}X + P_1 X + QX + \eta(X)\nu \nonumber\\
&=PX +QX+ \eta(X)\nu \label{eq3.1}
\end{align}
where $E = E_{\circ} \perp Rad(TN)$, $E$ is invariant and $PX \in \Gamma (E)$, $QX \in \Gamma(E^{\prime})$.
From (\ref{eq3.1}), we have
\begin{equation}\label{eq3.2}
\phi X = TX + wX,
\end{equation}
where $TX $ and $wX$ are tangential and transversal parts of $\phi X$, respectively. Also, $\phi(E^{\prime}) \neq E^{\prime}$.\\
on the other hand, for $Y \in \Gamma(E^{\prime})$, we obtain
\begin{equation}\label{eq3.3}
\phi Y = TY + wY
\end{equation}
such that $T Y \in \Gamma (E^{\prime})$ and $ wY \in \Gamma(S(TN^{\perp}))$.\\

Also, for $V \in \Gamma(tr(TN))$
\begin{equation}\label{eq3.4}
\phi V = BV + CV,
\end{equation}
where $BV$ and $CV$ are tangential and transversal parts of $\phi V$, respectively.

\begin{example}
Let $\tilde{N}= (R^{13}_6, \tilde{\rho})$ be an indefinite Sasakian manifold,  where  $\tilde{\rho}$ is of signature $ (-,-,-,+,+,+,-,-,-,+,+,+,+)$ with respect to the basis \{$\partial{x^1},\partial{x^2},\\\partial{x^3},\partial{x^4},\partial{x^5},\partial{x^6},\partial{y^1},\partial{y^2},\partial{y^3},\partial{y^4},\partial{y^5},\partial{y^6},\partial z$\}. If $(x_1,y_1,x_2,y_2,x_3,y_3,x_4,y_4,x_5,\\y_5,x_6,y_6,z)$ is the standard coordinate system of  $R^{13}_6$. Following definition (\ref{def2.4}), the triplet $(\bar \nabla = \bar \nabla^{\circ} + K, \tilde{\rho}, \phi)$ where $K$ satisfies (\ref{eq2.18}), defines an indefinite Sasakian statistical structure on $\tilde{ N}$. Consider a submanifold $N$ of $R^{13}_6 $ given by $ X=(0,u_5\cos\alpha,-u_5,-u_6,u_1\cosh\alpha,u_2\cosh\alpha, u_1\sinh\alpha -u_2, u_1+u_2\sinh\alpha,u_5\sin\alpha,\\ u_6\sin\alpha,\sin u_3\sinh u_4,\cos u_3\cosh u_4, u_7 )$. Then, $N$ is a SGL submanifold of an indefinite Sasakian statistical manifold $R^{13}_6$, for complete proof see \cite{D}.
\end{example}

\section{QS metric connection}
 Let $\tilde N$
be an indefinite Sasakian statistical manifold equipped with a QS metric connection $\tilde D$.\\
``For a Levi-Civita connection $\bar \nabla^{\circ}$ on an indefinite Sasakian statistical manifold $(\tilde N, \phi ,\tilde \rho)$ where $\bar \nabla^{\circ} = \frac{1}{2}\{\bar\nabla 
+ \bar\nabla^*\}$, we set
\begin{equation}\label{eq4.1}
\tilde D _X Y  =\bar\nabla_XY - K(X,Y) - \eta(X)\phi Y, 
\end{equation}
and 
\begin{equation}\label{eq4.2}
\tilde D _X Y  =\bar\nabla^*_XY + K(X,Y) - \eta(X)\phi Y, 
\end{equation}
for any $X,Y\in \Gamma({T\tilde N})$. Since $\bar\nabla $ and $\bar\nabla^*$ are torsion free, therefore from the relationship between dual connections, we obtain
\begin{equation}\label{eq4.3}
(\tilde D_X \tilde \rho)(Y,Z)= 0,
\end{equation}
and 
\begin{equation}\label{eq4.4}
\tilde T ^{\tilde D}(X,Y)= \eta(Y)\phi X-\eta(X)\phi Y,
\end{equation}
for any $X,Y,Z \in \Gamma({T\tilde N})$ where $\tilde T ^{\tilde D}$ is a torsion tensor of the connection $\tilde D$ and $\eta$ is a 1-form associated with the vector field $U$ on $\tilde M$ by $\pi(X) = \tilde{ \rho}(X,U)$. So, $\tilde D$ becomes a QS metric connection." 
\begin{theorem}
Let $(\tilde{N}, \bar{\nabla}, \bar{\nabla}^{*}, \tilde{\rho})$ be an indefinite statistical manifold with an almost contact metric structure $(\tilde{\rho}, \phi, \nu)$. Then $(\tilde{N},\bar{\nabla}, \bar{\nabla}^{*}, \tilde{\rho}, \phi, \nu)$ is said to be an indefinite  almost contact metric statistical manifold  $\tilde{N}$ with quarter QS metric connection if and only if
\begin{equation}\label{eq4.5}
(\tilde D_X \phi)Y = \tilde\rho(X,Y)\nu - \eta(Y)X
\end{equation}
\begin{equation}\label{eq4.6}
\tilde D_X \nu = -\phi X + \eta(\tilde D_X \nu)\nu
\end{equation}
\end{theorem}
Let $N$ be a contact SGL submanifold of  $\tilde N$ with QS metric connection $\tilde D$. Let $D$ be the induced linear connection on $N$ from $\tilde D$. Therefore the Gauss formula is as follows:
\begin{equation}\label{eq4.7}
\tilde D_XY = D_XY + \tilde h^l(X,Y) + \tilde h^s(X,Y),
\end{equation}
for any $X,Y \in \Gamma(TN)$, where $D_XY \in \Gamma(TN)$ and $\tilde h^l$, $\tilde h^s$ are lightlike second fundamental form and the screen second fundamental form of $N$, respectively. 
Now from (\ref{eq2.1}),(\ref{eq4.7}), (\ref{eq3.3}) in (\ref{eq4.1}), we get
\begin{equation}\label{eq4.8}
D_XY = \nabla_XY - \eta(X)TY - K(X,Y),
\end{equation}
\begin{equation}\label{eq4.9}
\tilde h^l(X,Y) =h^l(X,Y),\quad \quad
\tilde h^s(X,Y)=h^s(X,Y) -\eta(X) wY
\end{equation}
Further, using (\ref{eq4.3}), (\ref{eq4.7}), (\ref{eq4.8}) we have
\begin{equation}\label{eq4.10}
(D_X\tilde \rho)(Y,Z) =  \tilde \rho(\tilde h^l(X,Y),Z) + \tilde \rho(Y, \tilde h^l(X,Z)),
\end{equation}
and 
\begin{equation}\label{eq4.11}
T^D(X,Y) = \eta(Y) T X - \eta(X)TY.
\end{equation}
for any $X,Y,Z \in \Gamma(TN)$, where $T^D$ is torsion tensor of the induced connection $D$ on $N$. Subsequently the following result ensues.
\begin{theorem}
Let $N$ be a SGL submanifold of $\tilde{N}$ with a QS metric connection $\tilde{D}$. Then the connection $D$ on  $N$ is a QS non-metric connection.
\end{theorem}
Suppose that $\tilde h^l$ vanishes identically on $N$. Therefore
 we arrive to the following outcome:
\begin{theorem}
Let $N$ be a SGL submanifold of  $\tilde{N}$ with a QS metric connection $\tilde{D}$. Then the connection $D$ on  $N$ is also a QS metric connection if and only if $\tilde h^l$ vanishes identically on $N$.
\end{theorem}
Corresponding to QS metric connection $\tilde{D}$, the Weingarten formulae  are as below:
 
   \begin{equation}\label{eq4.12}
   	\tilde D_X N^{\prime} = -\tilde A_{N^{\prime}} X + \tilde\nabla^l_XN^{\prime} +\tilde D^s(X,N^{\prime}), 
   \end{equation}
   \begin{equation}\label{eq4.13}
   	\tilde D_X W =-\tilde A_W X +\tilde\nabla^s_X W +\tilde D^l(X,W),
   \end{equation}
   for any  $X, Y \in \Gamma(TN)$, $N^{\prime} \in \Gamma(ltr(TN))$ and $W\in \Gamma(S(TN^{\perp}))$.
 Using (\ref{eq2.4}),(\ref{eq2.5}) (\ref{eq4.12}),(\ref{eq4.13}) and (\ref{eq4.1}) and then equating the tangential and transversal parts, we derive
\begin{equation}\label{eq4.14}
\tilde A_{N^{\prime}} X = A_{N^{\prime}}X + K(X,N^{\prime}) +\eta(X)BW,
\end{equation}
\begin{equation}\label{eq4.15}
\tilde\nabla^l_X N^{\prime}= \nabla^l_XN^{\prime}, 
\end{equation}
\begin{equation}\label{eq4.16}
\tilde D^s(X,N^{\prime}) =  D^s(X,N^{\prime}) -\eta(X)CW,
\end{equation}
 Consider $P$ as the projection of $TN$ on $S(TN)$. Then, for any $X, Y \in \Gamma (TN)$, we have
\begin{equation}\label{eq4.17}
D_XPY = D^{\prime}_XPY + \tilde h^{\prime}(X,PY),\quad D_X\xi = -\tilde A^{\prime}_{\xi}X + \tilde\nabla^{\prime t}_X\xi,
\end{equation}
where $(D^{\prime}_XPY ,\tilde A^{\prime}_{\xi}X)$ and $(\tilde h^{\prime}(X,PY),\tilde\nabla^{\prime t}_X\xi) $ belong to $S(TN)$ and $Rad(TN)$ respectively. Thus we obtain
\begin{equation}\label{eq4.18}
D^{\prime}_XPY = \nabla^{\prime}_X PY - \eta(X)TPY - K(X,PY),
\end{equation}
\begin{equation}\label{eq4.19}
\tilde h^{\prime}(X,PY) = h^{\prime}(X,PY) 
\end{equation}
and
\begin{equation}\label{eq4.20}
\tilde A^{\prime}_{\xi}X = A^{\prime}_{\xi}X + \eta(X)T\xi + K(X,\xi),
\end{equation}
\begin{equation}\label{eq4.21}
\tilde\nabla^{\prime t}_X\xi = \nabla^{\prime t}_X\xi, 
\end{equation}
\begin{theorem}
Let $N$ be a contact SGL submanifold of $\tilde{N}$ with a QS metric connection $\tilde{D}$ with structure vector field $\nu$ tangent to $N$. Then 
\begin{enumerate}[(i)]
\item the distribution $E_{\circ}$ is integrable if and only if
\[
2\tilde\rho(Y,\phi X) = \eta(\tilde D_X\nu)\eta(Y) - \eta(\tilde D_Y\nu)\eta(X)
\]
\item the distribution $E_{\circ}\perp \nu$ is integrable if and only if
\[
\tilde\rho(D_X^{\prime}\phi Y - D_Y^{\prime}\phi X, TZ) = \tilde\rho(\tilde h^s(Y,\phi X) - \tilde h^s(X,\phi Y), wZ),
\]
\[
\tilde\rho(\tilde h^{\prime}(X,\phi Y),\phi N) = \tilde\rho(\tilde h^{\prime}(Y,\phi X),\phi N),
\]
for $X,Y \in\Gamma(E_{\circ})$.
\end{enumerate}
\end{theorem}
\begin{proof} (i) Suppose $E_{\circ}$ is integrable. Then, 
\[
\tilde\rho([X,Y],\nu) = 0, \;\;\;\; X,Y\in\Gamma(E_{\circ})
\]
Using metric connection $\tilde D$,

\begin{eqnarray}
\tilde\rho([X,Y],\nu)&=& \tilde\rho(\tilde D_XY- \tilde D_YX +\eta(Y)\phi(X) - \eta(X)\phi Y,\nu)\nonumber\\
&=& -\tilde\rho(Y,\tilde D_X\nu) + \tilde\rho(X,\tilde D_Y\nu)\nonumber
\end{eqnarray}
From equation (\ref{eq4.6}), we have
\begin{eqnarray}
\tilde\rho([X,Y],\nu) &=& \rho(Y,\phi X) -\eta(\tilde D_X\nu)\tilde\rho(Y,\nu) - \rho(X,\phi Y) + \eta(\tilde D_Y\nu)\tilde\rho(X,\nu)\nonumber\\
&=& 2\tilde\rho(Y,\phi X) - \eta(\tilde D_X\nu)\eta(Y) + \eta(\tilde D_Y\nu)\eta(X) \label{eq4.22}
\end{eqnarray}
(ii)  We know that  $E_{\circ}\perp \nu$ is integrable if and only if $[X,Y]\in\Gamma(E_{\circ})$ for all $X,Y\in\Gamma(E_{\circ})$, we get
\[
\tilde\rho([X,Y],Z) = \tilde\rho([X,Y],N) = 0
\]
for $Z \in\Gamma(E^{\prime})$ and $N\in\Gamma(ltr(TN))$.
\begin{eqnarray}
\tilde\rho([X,Y],Z) &=& \tilde\rho(\phi\tilde D_XY - \phi\tilde D_YX + \eta(Y)\phi^2 X - \eta(X)\phi ^2Y, \phi Z)\nonumber\\
&=& \tilde\rho(\tilde D_X\phi Y - \tilde D_Y\phi X,\phi Z) +2\eta(Y)\tilde\rho(X,\phi Z)  - 2 \eta(X)\rho(Y,\tilde\phi Z)\nonumber\\
&=& \tilde\rho(D_X \phi Y - D_Y\phi X, TZ) +\tilde\rho(\tilde h^s(X,\phi Y) - \tilde h^s(Y,\phi X),wZ) \label{eq4.23}
\end{eqnarray}
\begin{eqnarray}
\tilde\rho([X,Y],N^{\prime}) 
&=& \tilde\rho(\tilde D_X\phi Y - \tilde D_Y\phi X,\phi N^{\prime}) +2\eta(Y)\tilde\rho(X,\phi N^{\prime})  - 2 \eta(X)\rho(Y,\tilde\phi N^{\prime})\nonumber\\
&=& \tilde\rho(D_X \phi Y - D_Y\phi X,\phi N^{\prime})\nonumber\\
&=& \tilde\rho(\tilde h^{\prime}(X,\phi Y)- \tilde h^{\prime}( Y,\phi X), \phi N^{\prime})\label{eq4.24}
\end{eqnarray}
The result follows from equations (\ref{eq4.22}), (\ref{eq4.23}) and (\ref{eq4.24}).
\end{proof}
\begin{theorem}
Let $N$ be a contact SGL submanifold of $\tilde{N}$ with a QS metric connection $\tilde{D}$ with structure vector field $\nu$ tangent to $N$. Then 
\begin{enumerate}[(i)]
\item the distribution $E$ is not integrable. 
\item the distribution $E\perp \nu$ is integrable if and only if
\[
\tilde\rho(D_X^{\prime}\phi Y - D_Y^{\prime}\phi X, TZ) = \tilde\rho(\tilde h^s(Y,\phi X) - \tilde h^s(X,\phi Y), wZ),
\]
for $X,Y \in\Gamma(E)$, $Z\in\Gamma(E^{\prime})$.
\end{enumerate}
\end{theorem}
\begin{proof}Suppose $E$ is integrable then,
\[
\tilde\rho([X,Y],\nu) = 0, \;\;\;\; X,Y\in\Gamma(E)
\]
Now, using metric connection $\tilde D$ and equation (\ref{eq4.6}), we have
\begin{eqnarray}
\tilde\rho([X,Y],\nu)
&=& -\tilde\rho(Y,\tilde D_X\nu) + \tilde\rho(X,\tilde D_Y\nu)\nonumber\\
&=& 2\tilde\rho(Y,\phi X) - \eta(\tilde D_X\nu)\eta(Y) + \eta(\tilde D_Y\nu)\eta(X) 
\end{eqnarray}
which is not zero. So (i) proved.
Also, for $X,Y \in\Gamma(E)$, $Z\in\Gamma(E^{\prime})$
\begin{eqnarray}
\tilde\rho([X,Y],Z)
&=& \tilde\rho(\tilde D_X\phi Y - \tilde D_Y\phi X,\phi Z) \nonumber\\
&=& \tilde\rho(D_X \phi Y - D_Y\phi X, TZ) +\tilde\rho(\tilde h^s(X,\phi Y) - \tilde h^s(Y,\phi X),wZ) \label{eq4.23}
\end{eqnarray}
Hence, the proof is complete.
\end{proof}
\begin{theorem}
Let $N$ be a contact SGL submanifold of  $\tilde{N}$ with a QS metric connection $\tilde{D}$ with structure vector field $\nu$ tangent to $N$. Then 
\begin{enumerate}[(i)]
\item the distribution $E^{\prime}$ is integrable if and only if
\[
2\tilde\rho(Y,\phi X) = \eta(\tilde D_X\nu)\eta(Y) - \eta(\tilde D_Y\nu)\eta(X)
\]
\item the distribution $E^{\prime}\perp \nu$ is integrable if and only if
\[
D_YTZ - D_ZTY - \tilde A_{wZ}Y + \tilde A_{wY}Z 
\]
has no component in $\Gamma(E_{\circ})$ and $\Gamma(Rad(TM))$,
for $Y,Z \in\Gamma(E^{\prime}\perp \nu)$.
\end{enumerate}
\end{theorem}
\begin{proof}
For $X,Y \in\Gamma(E^{\prime})$
\[
\tilde\rho([X,Y],\nu) = 2\tilde\rho(Y,\phi X) - \eta(\tilde D_X\nu)\eta(Y) + \eta(\tilde D_Y\nu)\eta(X)
\]
using the integrability of the distribution $E^{\prime}$ along with the equations (\ref{eq4.4}), (\ref{eq4.6}), (i) holds.\\
Then, for $Y,Z \in\Gamma(E^{\prime}\perp \nu)$, $X\in\Gamma(E_{\circ})$  and  Using equations (\ref{eq4.4}), (\ref{eq4.5}), (\ref{eq2.14})
\[
\tilde\rho([Y,Z],X) = \tilde\rho(\tilde D_Y\phi Z - \tilde D_Z\phi Y, \phi X)
\]
Further, from equations (\ref{eq3.2}), (\ref{eq4.7}), (\ref{eq4.13}) we have
\begin{equation}\label{eq4.27}
\tilde\rho([Y,Z],X) = \rho(D_YTZ - D_ZTY - \tilde A_{wZ}Y + \tilde A_{wY}Z, \phi X)
\end{equation}
Also, $N^{\prime} \in\Gamma(ltr(TN))$
\begin{equation}\label{eq4.28}
\tilde\rho([Y,Z],N^{\prime}) = \rho(D_YTZ - D_ZTY - \tilde A_{wZ}Y + \tilde A_{wY}Z, \phi N^{\prime})
\end{equation}
Hence from the integrability of distribution $E^{\prime}\perp \nu$ and  equations (\ref{eq4.27}) and (\ref{eq4.28}), (ii) holds.
\end{proof}

\begin{theorem}
Let $N$ be a contact SGL submanifold of $\tilde{N}$ with a QS metric connection $\tilde{D}$ with structure vector field $\nu$ tangent to $N$. Then 
\begin{enumerate}[(i)]
\item the distribution $E$ is not parallel
\item the distribution $E\perp \nu$ is parallel if and only if
\[
\tilde \rho (D_XTZ, \phi Y) = \tilde{\rho}(\phi Y, \tilde A_{wz}X)  
\]
and 
\[
\tilde h^l(X,TZ) = - \tilde D^l(X,wZ)
\]
for $X,Z \in\Gamma(E\perp \nu)$.
\end{enumerate}
\end{theorem}
\begin{proof}
For $X, Y \in(E)$, Now, using equation (\ref{eq4.6}) 
\[
\tilde{\rho}(D_XY, V) = \tilde{\rho}(\tilde D_XY, V) =\tilde{\rho}(Y,\phi X)
\]
which is not zero as $E^{\prime}$ is non degenerate. So, (i) holds.\\

Now, using the equations (\ref{eq2.13}), (\ref{eq4.5}), we have
\[
\tilde{\rho}(D_X Y, Z) = \tilde{\rho}(\tilde D_XY,Z) = -\tilde{\rho}(\phi Y,\tilde D_X\phi Z)
\]
From equations (\ref{eq3.2}),(\ref{eq4.7}), we get
\[
\tilde{\rho}(D_X Y, Z) = -\tilde{\rho}(\phi Y, D_XTZ + \tilde h^l(X,TZ) - \tilde A_{wZ}X + \tilde D^l(X,wZ))
\]
Using the parallelism of the distribution $E\perp \nu$, the desired result follows.
\end{proof}
\begin{theorem}
Let $N$ be a contact SGL submanifold of  $\tilde{N}$ with a QS metric connection $\tilde{D}$ with structure vector field $\nu$ tangent to $N$. Then 
\begin{enumerate}[(i)]
\item the distribution $E^{\prime}$ is not parallel
\item the distribution $E^{\prime}\perp \nu$ is parallel if and only if
$D^{\prime}_YTZ- \tilde A_{wZ}Y $
has no component in $\Gamma(E_{\circ})$ and $\Gamma(Rad(TM))$,
for $Y,Z \in\Gamma(E^{\prime}\perp \nu)$.
\end{enumerate}
\end{theorem}
\begin{proof}
Suppose $E^{\prime}$ be a parallel distribution. Then,  we have
\[
\tilde\rho(D_YZ, V) = 0,
\]
for $Y,Z \in\Gamma(E^{\prime})$. Now, using equation (\ref{eq4.6}) 
\[
\tilde\rho(D_YZ, V) = \tilde\rho(\tilde D_YZ, V)=\tilde\rho(Z,\phi Y)
\]
which is not zero. This is a contradiction as $E^{\prime}$ is non-degenerate. So (i) holds.
Also, $E^{\prime}\perp \nu$ be a  parallel distribution. Then, for $Y,Z \in\Gamma(E^{\prime}\perp \nu)$, $D_YZ \in\Gamma(E^{\prime}\perp \nu)$. Hence we have, 
\[
\tilde\rho(D_YZ, X) =\tilde\rho(D_YZ, N)= 0,
\]
for $X\in\Gamma(E_{\circ})$ and $N\in\Gamma(ltr(TM))$.Using equations (\ref{eq2.13}),  (\ref{eq4.5}), (\ref{eq3.3}), we have
\begin{eqnarray}
\tilde\rho(D_YZ,X) &=&\tilde\rho(\tilde D_YZ,X)=\tilde\rho(\tilde D_Y\phi Z -\tilde\rho(Y,Z)V +\eta(Z)Y, \phi X) \nonumber\\
&=& \tilde(D_YTZ-\tilde A_{wZ}Y, \phi X) \label{eq4.29}
\end{eqnarray}
Also
\begin{equation}\label{eq4.30}
\tilde\rho(D_YZ,N)= \tilde(D_YTZ-\tilde A_{wZ}Y, \phi N)
\end{equation}
The result follows from (\ref{eq4.29}), (\ref{eq4.30}).
\end{proof}

\section{Geodesic SGL Submanifolds}
\begin{definition}
  ``A SGL submanifold of $\tilde{N}$ with a QS metric connection is said to be a $E$- geodesic  if $\tilde h(X,Y) = 0$  for \;  $X, Y \in \Gamma(E)$. \\
  
 $N$ is  said to be $E$-geodesic if $\tilde h^l(X,Y) = 0$, $\tilde h^s(X,Y) = 0$ for any $X, Y \in \Gamma(E).$ \\
  
  Also, $N$ is said to be mixed geodesic if $\tilde h(X,Y) = 0$,  for any $X \in \Gamma(E)$ and $Y \in \Gamma(E^{\prime}\perp\nu)$.  "
  \end{definition}

\begin{theorem}
Let $N$ be a SGL submanifold of  $\tilde{N}$ with a QS metric connection $\tilde{D}$. If the distribution $E\perp\nu$ defines totally geodesic foliation in $\tilde{N}$  if and only if $N$ is $E\perp\nu$-geodesic and $E\perp\nu$ is parallel respect to $D$ on $N$.
\end{theorem} 
\begin{proof}
We know that $E\perp\nu$ defines a totally geodesic foliations in $\tilde{N}$ if and only if 
\[
 \tilde \rho(\tilde D_XY,\xi) = \tilde \rho(\tilde D_XY,W) =\tilde\rho(\tilde D_XY,Z)= 0
 \] 
for $X,Y \in \Gamma(E\perp\nu)$, $\xi \in \Gamma Rad(TN)$, $Z \in \Gamma(E^{\prime})$ and $W\in \Gamma(S(TN^{\perp}))$.
\begin{equation}\label{eq5.1}
\tilde \rho(\tilde D_XY,\xi) = \tilde \rho(\tilde h^l(X,Y),\xi)
\end{equation}
Also
\begin{equation}\label{eq5.2}
\tilde \rho(\tilde D_XY,W) = \tilde\rho(\tilde h^s(X,Y), W) 
\end{equation}
Now
\begin{equation}\label{eq5.3}
\tilde\rho(\tilde D_XY,Z) = \tilde\rho(D_XY, Z)
\end{equation}
From equations (\ref{eq5.1}),(\ref{eq5.2}) and (\ref{eq5.3}), the proof is complete. 
\end{proof}

\begin{theorem}
Let $N$ be a SGL submanifold of $\tilde{N}$ with a QS metric connection $\tilde{D}$. If $N$ is mixed geodesic if and only if the following satisfy:
\begin{enumerate}
\item  $\tilde h^l(X,T Z) =- D^l(X,wZ)$,
\item $\tilde\rho(\tilde A_{wZ}X - D_XT Z, BW) = \tilde\rho(\tilde h^s(X,T Z) + \tilde\nabla^s_XwZ, CW)$
\end{enumerate}
for $X \in \Gamma(E)$,$Z \in \Gamma(E^{\prime}\perp\nu)$ and $W\in \Gamma(S(TN^{\perp}))$.
\end{theorem} 
\begin{proof} For  $X \in \Gamma(E)$, $Z \in \Gamma(E^{\prime}\perp\nu)$ and $\xi \in \Gamma (Rad(TN))$
\[
\tilde\rho(\tilde D_XZ,\xi)=0
\]
using equation (\ref{eq4.7}) along with the mixed geodesicity  of $N$.
From the equations (\ref{eq2.13}),(\ref{eq4.5}),(\ref{eq4.7}), we have
\begin{eqnarray}
\tilde\rho(\tilde D_XZ,\xi)&=& \tilde\rho(\phi\tilde D_XZ,\phi\xi)+\eta(\tilde D_XZ)\eta(\xi)\nonumber\\ &=&\tilde\rho(\tilde D_X\phi Z-\tilde\rho(X,Z)\nu+\eta(Z)X,\phi \xi)+ \eta(\tilde D_XZ)\eta(\xi)\nonumber\\ &=& \tilde\rho(\tilde h^l(X,TZ)+ \tilde D^l(X,wZ),\xi) \nonumber
\end{eqnarray}
result (i) holds.

Also,
\[
\tilde\rho(\tilde D_XZ, W) = \tilde\rho(\phi\tilde D_XZ,\phi W)+\eta(\tilde D_XZ)\eta(W)
\]
\begin{equation}\label{eq5.4}
\tilde\rho(\tilde D_XZ, W) = \tilde\rho(-\tilde A_{wZ}X + D_X T Z + \tilde h^s(X,T Z) + \tilde\nabla^s_XwZ, BW + CW)
\end{equation}
The result follows from the hypothesis along with equation (\ref{eq5.4}).
\end{proof}
\begin{lemma}
Let $N$ be a SGL submanifold of $\tilde{N}$ with a QS metric connection $\tilde{D}$. Then
\begin{equation}\label{eq5.5}
D_XT Y - \tilde A_{wY}X = T D_XY + B\tilde h^s(X,Y) -\eta(Y)X +\rho(X,Y)\nu
\end{equation}
\begin{equation}\label{eq5.6}
\tilde h^l(X,T Y) + \tilde D^l(X,wY) = C\tilde h^l(X,Y)
\end{equation}
\begin{equation}\label{eq5.7}
\tilde h^s(X,T Y) +\tilde\nabla^s_XwY = wD_XY +C\tilde h^s(X,Y)
\end{equation}
for $X,Y \in \Gamma(TN)$.
\end{lemma}  
\begin{proof}
Differentiating (\ref{eq3.2}) with respect to $Y\in\Gamma(TN)$ and using equations (\ref{eq4.5}), (\ref{eq4.6}), we have
\begin{eqnarray}
D_XT Y &+& \tilde h^l(X,T Y) + \tilde h^s(X,T Y) - \tilde A_{wY}X + \tilde\nabla^s_XwY + \tilde D^l(X,wY)\nonumber \\ &=& T D_XY + wD_XY + C\tilde h^l(X,Y) + B\tilde h^s(X,Y) + C\tilde h^s(X,Y)\nonumber\\ &-&\eta(Y)X +\rho(X,Y)\nu \nonumber
\end{eqnarray}
By comparing tangential, lightlike transversal and screen transversal parts, we get (\ref{eq5.5}), (\ref{eq5.6}) and (\ref{eq5.7}). This completes the proof.
\end{proof}
\begin{theorem}
Let $N$ be a SGL submanifold of $\tilde{N}$ with a QS metric connection $\tilde{D}$. If $N$ is mixed geodesic if and only if the following satisfies:
\begin{enumerate}
\item $\tilde h^l(X,T Z)= -D^l(X,wZ)$
\item $w(\tilde A_{wZ}X - D_XT Z) = C(\tilde h^s(X,T Z) + \tilde\nabla^s_XwZ)$
\end{enumerate}
for $X \in \Gamma(E)$, $Z \in \Gamma(E^{\prime}\perp \nu)$.
\end{theorem}
\begin{proof}
From equations (\ref{eq2.14}) and (\ref{eq4.5}), we have
\begin{eqnarray}
\phi(D_X\phi Z) = -D_XZ+\eta(D_XZ)\nu-\tilde h^l(X,Z)-\tilde h^s(X,Z)-\eta(Z)\phi X\nonumber
\end{eqnarray}
Using equation (\ref{eq3.3}),(\ref{eq4.7}), (\ref{eq4.13}), we get
\begin{eqnarray}
-\tilde h^l(X,Z)-\tilde h^s(X,Z)&=& -wA_{wZ}X + w D_XT Z + C\tilde h^l(X,T Z) + C\tilde h^s(X,T Z) \nonumber\\ &&+ C\tilde\nabla^s_XwZ + C\tilde D^l(X,wZ)  \nonumber
\end{eqnarray}
From the mixed geodesicity of $N$ and comparing tangential and transversal part, result holds.
\end{proof}

\begin{theorem}
Let $N$ be a SGL submanifold of $\tilde{N}$ with a QS metric connection $\tilde{D}$. Then for $X \in \Gamma(E_{\circ})$, $Z \in \Gamma(E^{\prime})$, we have
\[
D_XZ = -T D_XT Z + T \tilde A_{wZ}X - B\tilde h^s(X,T Z) - B\tilde\nabla^s_XwZ + \eta(Z)\phi X + \eta(D_XZ)\nu
\]
\end{theorem}
\begin{proof} For $X \in \Gamma(E_{\circ})$, $Z \in \Gamma(E^{\prime})$
\begin{eqnarray}
\tilde D_X Z  &=& -T D_XT Z -w D_XT Z -C\tilde h^l(X,T Z) - B\tilde h^s(X,T Z) - C\tilde h^s(X,T Z)\nonumber\\ &&+ T\tilde A_{wZ}X +wA_{wZ}X - B\tilde\nabla^s_XwZ - C\tilde\nabla^s_XwZ - C\tilde D^l(X,wZ) \nonumber\\ &&+ \eta(Z)\phi X + \eta(D_XZ)\nu  \nonumber
\end{eqnarray}
from equations (\ref{eq2.14}), (\ref{eq4.5}), (\ref{eq3.3}) and (\ref{eq3.4}).
\end{proof}
By taking tangential part, we get the assertion.
\vspace{0.1in}
\section{Concluding Remarks}

The investigation of statistical manifolds has gained considerable attention due to their diverse applications across various domains such as statistical inference, neural networks, image processing, clustering, etc. However, despite these advancements the integration of lightlike geometry with statistical and contact structures remains relatively underdeveloped.Thus we have conducted a detailed analysis of the structural characteristics of SGL submanifolds within the framework of an indefinite Sasakian statistical manifold and provides novel perspectives on the relationship between statistical geometry and indefinite contact structures.The findings presented in this work open new avenues for extending similar analyses to a broader class of geometric settings, particularly within complex and contact metric manifolds endowed with specialized connections. Such extensions may lead to deeper understanding and novel applications in differential geometry, theoretical physics, and data-driven learning models.

\end{document}